\documentclass[graybox]{svmult}

\usepackage{mathptmx}       
\usepackage{helvet}         
\usepackage{courier}        
\usepackage{type1cm}        
\usepackage{makeidx}         
\usepackage{graphicx}        
\usepackage{multicol}        
\usepackage[bottom]{footmisc}

\usepackage{amssymb}

\usepackage[french,english]{babel}

\makeindex             

\begin{document}

\title*{The Tower of Hanoi and Finite Automata}
\author{Jean-Paul Allouche and Jeff Shallit}

\institute{Jean-Paul Allouche \at CNRS, IMJ-PRG, Sorbonne Universit\'e
4 Place Jussieu 75005 Paris, France, \email{jean-paul.allouche@imj-prg.fr}
\and Jeff Shallit \at School of Computer Science, University of Waterloo, Waterloo, 
Ontario  N2L 3G1, Canada, \email{shallit@uwaterloo.ca}}

\maketitle

\abstract{Some of the algorithms for solving the Tower of Hanoi puzzle can
be applied ``with eyes closed'' or ``without memory''. Here we survey the solution
for the classical Tower of Hanoi that uses finite automata,
as well as some variations on the original puzzle.
In passing, we obtain a new result on morphisms generating the classical and
the lazy Tower of Hanoi.}

\section{Introduction}

A huge literature in mathematics and theoretical computer science deals with 
the Tower of Hanoi and generalizations. The reader can look at the references given 
in the bibliography of the present paper, but also at the papers cited in these 
references (in particular in \cite{AD, Hinz1}). A very large bibliography was 
provided by Stockmeyer \cite{Stockmeyer}.  Here we present a survey of the relations
between the Tower of Hanoi and monoid morphisms or finite automata.
We also give a new result on morphisms generating the classical and the lazy Tower 
of Hanoi (Theorem~\ref{newHanoi}).

Recall that the Tower of Hanoi puzzle has three pegs, labeled I, II, III,
and $N$ disks of radii $1, 2, \ldots, N$. At the beginning the disks are placed
on peg I, in decreasing order of size (the smallest disk on top). A move consists
of taking the topmost disk from one peg and moving it to another peg, with the
condition that no disk should cover a smaller one. The purpose is to transfer all
disks from the initial peg to another one (where they are thus in decreasing order
as well).

The usual ({\it recursive}) approach for solving the Hanoi puzzle consists in 
noting that, in order to move $(N+1)$ disks from a peg to another, it is necessary 
and sufficient first to move the smallest $N$ disks to the third peg, then to move 
the largest disk to the now empty peg, and finally to transfer the smallest $N$ 
disks on that third peg. An easy induction shows that the number of moves for $N$ 
disks is thus $2^N-1$ and that it is optimal.

Applying this recursive algorithm with a small number of disks (try it with $3$ disks), 
shows that it transfers $1$ disk (the smallest) from peg I to peg II; then continuing 
the process, the sub-tower consisting of the disks of radii $1$ and $2$ will be 
reconstructed on peg III; and the sub-tower consisting of the disks of radii $1$, $2$ 
and $3$ will be reconstructed on peg II.
More generally, let ${\cal S}_N$ be the sequence of moves that transfers the tower
with the smallest $N$ disks from peg I to peg II if $N$ is odd, and from peg I to 
peg III if $N$ is even. Then, for any positive integer $k \leq N$, the sequence 
${\cal S}_N$ begins with the sequence ${\cal S}_k$. In other words, there exists 
an infinite sequence of moves ${\cal S}_{\infty}$, such that, for any integer $N$, 
the first $(2^N-1)$ moves of ${\cal S}_{\infty}$ solve the Hanoi puzzle by moving the 
tower of $N$ disks from peg I to peg II or III according to whether
$N$ is odd or even.

We let $a$, $b$, $c$ denote the moves that take the topmost disk from peg I to
peg II, resp. from peg II to peg III, resp. from peg III to peg I. 
Let $\overline{a}$, $\overline{b}$, $\overline{c}$ be the inverse moves (e.g., 
$\overline{c}$ moves the topmost disk from peg I to peg III). Then, as the reader
can easily check
$$
{\cal S}_{\infty} = a \ \overline{c} \ b \ a \ c \ \overline{b} \ a \ \overline{c} \ b \ 
\overline{a} \ c \ b \ a \ \overline{c} \ b \ a \ \cdots
$$

\begin{remark}

Playing with this algorithm leads to the discovery (and to the proof) 
of a ``simpler'' algorithm for the puzzle's solution, where

\medskip

- the first, third, fifth, etc., moves only concern the smallest disk, 
which moves circularly from peg I to peg II, from peg II to peg III, from 
peg III to peg I, and so forth;

\smallskip

- the second, fourth, sixth, etc., moves leave the smallest disk
fixed on its peg.  Hence, they consist in looking at the topmost disk of 
each of the two other pegs, and in moving the smaller to cover the larger.

\medskip

We note that this ``simpler'' algorithm cannot be performed ``without 
memory'' nor ``with eyes closed'' (i.e., without looking at the pegs):
namely at the even steps, we need to know the sizes of the topmost disks and 
compare them.  The next section 
addresses the question of finding an algorithm that can
be applied ``with bounded memory'' and ``with eyes closed''.

The ``simpler'' algorithm where the smallest disk moves circularly every second 
move is attributed to Raoul Olive, the nephew of \'Edouard Lucas in \cite{Lucas}. 
Also, reconstructing the Tower of Hanoi on peg II or peg III according to the parity 
of the number of disks can be seen as a ``dual'' of the strategy of R. Olive, 
where the smallest disk moves circularly either clockwise or counter-clockwise,
according to the parity of the number of disks and the desired final peg where
the Tower of Hanoi is reconstructed.
\end{remark}

\begin{remark}\label{variations}
Several variations on this game have been
introduced: the cyclic Tower of Hanoi (using only the
moves $a$, $b$, and $c$ in the notation above), the
lazy Tower of Hanoi
(using only the moves $a$, $\overline{a}$, $b$, $\overline{b}$), the
colored Tower of Hanoi,
Antwerpen Towers, $d$ pegs instead of $3$ pegs, etc. 
There are also variations studied 
in cognitive psychology: the Tower of Hanoi itself \cite{Simon}, the Tower of
London \cite{Shallice}, and the Tower of Toronto \cite{SaintCyr}. We do not resist
to propose a modest contribution to the world of variations on the Tower of
Hanoi, in honor of the organizers of the Symposium 
\selectlanguage{french}
``La \og Tour d'Hano\"{\i} \fg, 
un casse-t\^ete math\'ematique d'\'Edouard Lucas (1842-1891)''.
\selectlanguage{english}
Start with three pegs, and disks indexed by a given word on the usual Latin alphabet. 
Move as usual the topmost disk from a peg to another, the rule being that no two 
consecutive vowels should appear. Start from {\tt HANOI}. Well,
{\tt O} and {\tt I} 
are already consecutive. Let us say that ``{\tt O} = oh = Zero = {\tt Z}'', and let 
us thus replace {\tt HANOI} with {\tt HANZI}.  Here are a few permitted moves:
starting with {\tt HANZI}, we get successively
$$
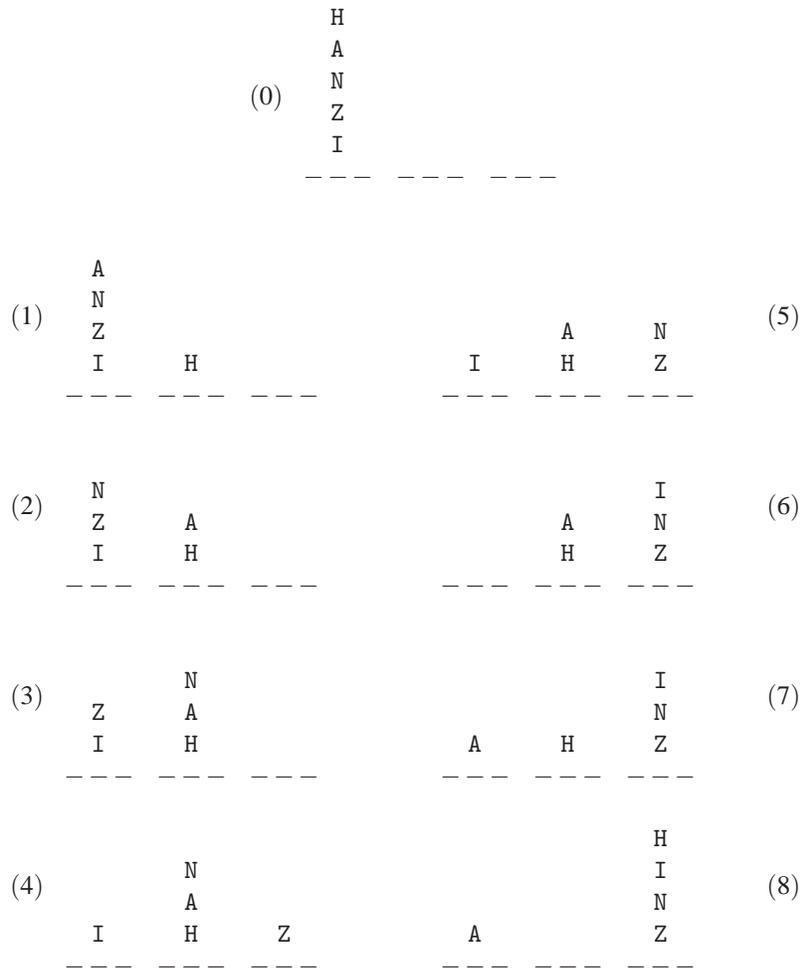
\begin{array}{lllll}
  &(0) 
  &\begin{array}{cc}
     &{\tt H} \\
     &{\tt A} \\
     &{\tt N} \\
     &{\tt Z} \\
     &{\tt I} \\
     &{---} \\
   \end{array}
  &\begin{array}{cc}
     &{\tt } \\
     &{\tt } \\
     &{\tt } \\
     &{\tt } \\
     &{\tt } \\
     &{---} \\
   \end{array}
  &\begin{array}{cc}
     &{\tt } \\
     &{\tt } \\
     &{\tt } \\
     &{\tt } \\
     &{\tt } \\
     &{---} \\
   \end{array}
 \\
\end{array}
$$
$$
\begin{array}{lllll}
  &(1)
  &\begin{array}{cc}
     &{\tt } \\
     &{\tt A} \\
     &{\tt N} \\
     &{\tt Z} \\
     &{\tt I} \\
     &{---} \\
   \end{array}
  &\begin{array}{cc}
     &{\tt } \\
     &{\tt } \\
     &{\tt } \\
     &{\tt } \\
     &{\tt H} \\ 
     &{---} \\
   \end{array}
  &\begin{array}{cc}
     &{\tt } \\
     &{\tt } \\
     &{\tt } \\
     &{\tt } \\
     &{\tt } \\
     &{---} \\
   \end{array}
\\
  &(2)
  &\begin{array}{cc}
     &{\tt } \\
     &{\tt } \\
     &{\tt N} \\
     &{\tt Z} \\
     &{\tt I} \\
     &{---} \\
   \end{array}
  &\begin{array}{cc}
     &{\tt } \\
     &{\tt } \\
     &{\tt } \\
     &{\tt A} \\
     &{\tt H} \\
     &{---} \\
   \end{array}
  &\begin{array}{cc}
     &{\tt } \\
     &{\tt } \\
     &{\tt } \\
     &{\tt } \\
     &{\tt } \\
     &{---} \\
   \end{array}
\\
  &(3)
  &\begin{array}{cc}
     &{\tt } \\
     &{\tt } \\
     &{\tt } \\
     &{\tt Z} \\
     &{\tt I} \\
     &{---} \\
   \end{array}
  &\begin{array}{cc}
     &{\tt } \\
     &{\tt } \\
     &{\tt N} \\
     &{\tt A} \\
     &{\tt H} \\
     &{---} \\
   \end{array}  
  &\begin{array}{cc}
     &{\tt } \\
     &{\tt } \\
     &{\tt } \\
     &{\tt } \\
     &{\tt } \\
     &{---} \\
   \end{array}
\\
  &(4)
  &\begin{array}{cc}
     &{\tt } \\
     &{\tt } \\
     &{\tt } \\
     &{\tt } \\
     &{\tt I} \\
     &{---} \\
   \end{array}
  &\begin{array}{cc}
     &{\tt } \\
     &{\tt } \\
     &{\tt N} \\
     &{\tt A} \\
     &{\tt H} \\
     &{---} \\
   \end{array} 
  &\begin{array}{cc}
     &{\tt } \\
     &{\tt } \\
     &{\tt } \\
     &{\tt } \\
     &{\tt Z} \\
     &{---} \\
   \end{array}
\\
\end{array}
\hskip 1.2truecm
\begin{array}{lllll}
  &\begin{array}{cc}
     &{\tt } \\
     &{\tt } \\
     &{\tt } \\
     &{\tt } \\
     &{\tt I} \\
     &{---} \\
   \end{array}
  &\begin{array}{cc}
     &{\tt } \\
     &{\tt } \\
     &{\tt } \\
     &{\tt A} \\
     &{\tt H} \\
     &{---} \\
   \end{array}
  &\begin{array}{cc}
     &{\tt } \\
     &{\tt } \\
     &{\tt } \\
     &{\tt N} \\
     &{\tt Z} \\
     &{---} \\
   \end{array}
  &(5) 
\\
  &\begin{array}{cc}
     &{\tt } \\
     &{\tt } \\
     &{\tt } \\
     &{\tt } \\
     &{\tt } \\
     &{---} \\
   \end{array}
  &\begin{array}{cc}
     &{\tt } \\
     &{\tt } \\
     &{\tt } \\
     &{\tt A} \\
     &{\tt H} \\
     &{---} \\
   \end{array}
  &\begin{array}{cc}
     &{\tt } \\
     &{\tt } \\
     &{\tt I} \\
     &{\tt N} \\
     &{\tt Z} \\
     &{---} \\
   \end{array}
  &(6) 
\\
  &\begin{array}{cc}
     &{\tt } \\
     &{\tt } \\
     &{\tt } \\
     &{\tt } \\
     &{\tt A} \\
     &{---} \\
   \end{array}
  &\begin{array}{cc}
     &{\tt } \\
     &{\tt } \\
     &{\tt } \\
     &{\tt } \\
     &{\tt H} \\
     &{---} \\
   \end{array}
  &\begin{array}{cc}
     &{\tt } \\
     &{\tt } \\
     &{\tt I} \\
     &{\tt N} \\
     &{\tt Z} \\
     &{---} \\
   \end{array}
  &(7) 
\\
  &\begin{array}{cc}
     &{\tt } \\
     &{\tt } \\
     &{\tt } \\
     &{\tt } \\
     &{\tt A} \\
     &{---} \\
   \end{array}
  &\begin{array}{cc}
     &{\tt } \\
     &{\tt } \\
     &{\tt } \\
     &{\tt } \\
     &{\tt } \\
     &{---} \\    
   \end{array}
  &\begin{array}{cc}
     &{\tt } \\
     &{\tt H} \\
     &{\tt I} \\
     &{\tt N} \\
     &{\tt Z} \\
     &{---} \\
   \end{array} \hskip .8truecm
  &(8) 
\\
\end{array}
$$

\bigskip
\medskip

\noindent
{\it Exercise: concoct two variations giving respectively} 

$$
\begin{array}{cc}
     &{\tt H} \\
     &{\tt A} \\
     &{\tt N} \\
     &{\tt O} \\
     &{\tt I} \\
\end{array} \ \ \ \ \
\longrightarrow \ \
\ldots \ \ 
\longrightarrow \ \ \ \ \
\begin{array}{cc}
     &{\tt D} \\
     &{\tt A} \\
     &{\tt N} \\
     &{\tt E} \\
     &{\tt K} \\
\end{array} 
\hskip 1truecm
\mbox{\it and}
\hskip 1truecm  
\begin{array}{cc}
     &{\tt H} \\
     &{\tt A} \\
     &{\tt N} \\
     &{\tt O} \\
     &{\tt I} \\
\end{array}  \ \ \ \ \
\longrightarrow \ \
\ldots \ \
\longrightarrow \ \ \ \ \
\begin{array}{cc}
     &{\tt G} \\
     &{\tt A} \\
     &{\tt U} \\
     &{\tt Z} \\
     &{\tt Y} \\
\end{array}
$$
\end{remark}

\section{The infinite sequence of moves ${\cal S}_{\infty}$
and an easy way of generating it}

We will focus on the infinite sequence ${\cal S}_{\infty}$ defined above.
It is a sequence on six symbols, i.e., a sequence over the $6$-letter alphabet
$\{a, b, c, \overline{a}, \overline{b}, \overline{c}\}$).
This sequence can also be seen as a sequence of moves that ``tries'' to
reconstruct a Tower of Hanoi with infinitely many disks, by reconstructing the 
sub-towers with the smallest $N$ disks for $N = 1, 2, 3, \ldots$ on peg II or III
acording to the parity of $N$.

Group the letters of ${\cal S}_{\infty}$ pairwise, and write this sequence 
of pairs of letters just under the sequence ${\cal S}_{\infty}$:
$$
\begin{array}{cccccccccc}
&a \ &\overline{c} \ &b \ &a \ &c \ &\overline{b} \ &a  \ &\overline{c} \ 
&\cdots
\\
&(a \ \overline{c}) \ &(b \ a) \ &(c \ \overline{b}) \ &(a \ \overline{c})
\ &(b \ \overline{a}) \ &(c \ b) \ &(a \ \overline{c}) \ &(ba) \  &\cdots
\\
\end{array}
$$
We observe that under any of the six letters we always find the same pair 
of letters, e.g., there always is an $(a\overline{c})$ under an $a$. More 
precisely, if we associate a $2$-letter word with each letter in 
$\{a, b, c, \overline{a}, \overline{b}, \overline{c}\}$ 
as follows
$$
\begin{array}{lllllll}
&a &\longrightarrow& a \overline{c} \hskip 1truecm &\overline{a} &\longrightarrow& a c \\
&b &\longrightarrow& c \overline{b}                &\overline{b} &\longrightarrow& c b \\
&c &\longrightarrow& b \overline{a}                &\overline{c} &\longrightarrow& b a \\
\end{array}
$$
we can obtain the infinite sequence ${\cal S}_{\infty}$ by starting with $a$
and iterating the map above, where the image of a word is obtained by ``gluing''
({\em concatenating}) together the images of letters of that word:
$$
a \longrightarrow a \ \overline{c} \longrightarrow a \ \overline{c} \ b \ a
\longrightarrow a \ \overline{c} \ b \ a \ c \ \overline{b} \ a \ \overline{c} 
\longrightarrow \cdots
$$
This result was proven in \cite{ABS}. We give more details below.

\begin{definition}
Let ${\cal A}$ be an {\em alphabet}, i.e., a finite set. A {\em word} on ${\cal A}$
is a finite sequence of symbols from ${\cal A}$ (possibly empty). The set of all
words over ${\cal A}$ is denoted by ${\cal A}^*$. The {\em length} of a word is the 
number of symbols that it contains (the length of the empty word $\epsilon$ is $0$). 
The {\em concatenation} of two words $a_1 a_2 \cdots a_r$ and 
$b_1 b_2 \cdots b_s$ of lengths $r$ and $s$, respectively, is the word 
$a_1 a_2 \cdots a_r b_1 b_2 \cdots b_s$ of length $r+s$ obtained by 
gluing them in order. The set ${\cal A}^*$ equipped with concatenation is
called the {\em free monoid} generated by ${\cal A}$.

A sequence of words $u_{\ell}$ of ${\cal A}^*$ is said to {\em converge}
to the infinite sequence $(a_n)_{n \geq 0}$ on ${\cal A}$ if the length
of the largest prefix of $u_{\ell}$ that coincides with the prefix of 
$(a_n)_{n \geq 0}$ of the same length tends to infinity with $\ell$.
\end{definition}

\begin{remark}
It is easy to see that ${\cal A}^*$ equipped with concatenation is indeed
a monoid: concatenation is associative, and the empty word $\epsilon$ is the identity
element. This monoid is {\em free}: this means intuitively that there are no
relations between elements other than the relations arising from
the associative property
and the fact that the empty word is the identity element. In particular this
monoid is not {\em commutative} if ${\cal A}$ has at least two distinct
elements.
\end{remark}

\begin{definition}
Let ${\cal A}$ and ${\cal B}$ be two alphabets. A {\em morphism} from ${\cal A}^*$ 
to ${\cal B}^*$ is a map $\varphi$ from ${\cal A}^*$ to ${\cal B}^*$, such that, 
for any two words $u$ and $v$, one has $\varphi(uv) = \varphi(u)\varphi(v)$.
A morphism of ${\cal A}^*$ is a morphism from ${\cal A}^*$ to itself.

If there exists a positive integer $k$ such that $\varphi(a)$ has length $k$
for all $a \in {\cal A}$,
the morphism $\varphi$ is said to be {\em $k$-uniform}.

\end{definition}

\begin{remark}
A morphism $\varphi$ from ${\cal A}^*$ to ${\cal B}^*$ is determined by the 
values of $\varphi(a)$ for $a \in {\cal A}$. Namely, if the word $u$ is equal 
to $a_1 a_2 \cdots a_n$ with $a_j \in {\cal A}$, then $\varphi(u) = 
\varphi(a_1) \varphi(a_2) \cdots \varphi(a_n)$.
\end{remark}

\begin{definition}
An infinite sequence $(a_n)_{n \geq 0}$ taking values in the alphabet 
${\cal A}$ is said to be {\em pure morphic} if there exist a morphism
$\varphi$ of ${\cal A}^*$ and a word $x \in {\cal A}^*$ such that
\begin{itemize}
\item the word $\varphi(a_0)$ begins with $a_0$,
      (there exists a word $x$ such that $\varphi(a_0) = a_0 x$);
\item iterating $\varphi$ starting from $x$ never gives the empty word
      (for each integer $\ell$, $\varphi^{\ell}(x) \neq \epsilon$);
\item the sequence of words $\varphi^{\ell}(a_0)$ converges to the
      sequence $(a_n)_{n \geq 0}$ when $\ell \to \infty$.
\end{itemize}
\end{definition}

\begin{remark}
It is immediate that 
$$
\begin{array}{lll}
\varphi(a_0)   &=& a_0 x \\
\varphi^2(a_0) &=& \varphi(\varphi(a_0)) = \varphi(a_0 x) = \varphi(a_0) \varphi(x)
= a_0 x \varphi(x) \\
\varphi^3(a_0) &=& \varphi(\varphi^2(a_0)) = \varphi(a_0 x \varphi(x)) =
\varphi(a_0) \varphi(x) \varphi^2(x) = a_0 x \varphi(x) \varphi^2(x) \\
&\vdots& \\
\varphi^{\ell}(a_0) &=& a_0 x \varphi(x) \varphi^2(x) \cdots \varphi^{\ell-1}(x) \\
&\vdots& \\
\end{array}
$$
\end{remark}

\begin{definition}
An infinite sequence $(a_n)_{n \geq 0}$ with values in the alphabet
${\cal A}$ is said to be {\em morphic} if there exist an alphabet 
${\cal B}$ and an infinite sequence $(b_n)_{n \geq 0}$ on ${\cal B}$
such that
\begin{itemize}
\item the sequence $(b_n)_{n \geq 0}$ is pure morphic;
\item there exists a $1$-uniform morphism from ${\cal B}^*$ to
      ${\cal A}^*$ sending the sequence $(b_n)_{n \geq 0}$ to
      the sequence $(a_n)_{n \geq 0}$ (i.e., the sequence 
      $(a_n)_{n \geq 0}$ is the pointwise image of $(b_n)_{n \geq 0}$).
\end{itemize}
If the morphism making $(b_n)_{n \geq 0}$ morphic is $k$-uniform,
then the sequence $(a_n)_{n \geq 0}$  is said to be {\em $k$-automatic}.
The word ``automatic'' comes from the fact that the sequence $(a_n)_{n \geq 0}$
can be generated by a finite automaton (see \cite{AS} for more details on this topic).
\end{definition}

\begin{remark}
A morphism $\varphi$ of ${\cal A}^*$ can be extended to infinite sequences
with values in ${\cal A}$ by defining
$\varphi((a_n)_{n \geq 0}) = \varphi(a_0 a_1 a_2 \cdots) := 
\varphi(a_0) \varphi(a_1) \varphi(a_2) \cdots$

\noindent
It is easy to see that a pure morphic sequence is a {\em fixed point} of 
(the extension to infinite sequences of) some morphism: actually, with the
notation above, it is {\em the} fixed point of $\varphi$ beginning with 
$a_0$. A pure morphic sequence is also called an {\em iterative fixed point} 
of some morphism (because of the construction of that fixed point), while a 
morphic sequence is the pointwise image of an iterative fixed point of some 
morphism, and a $k$-automatic sequence is the pointwise image of the iterative 
fixed point of a $k$-uniform morphism.
\end{remark}

We can now state the following theorem \cite{ABS,AD}:

\begin{theorem}\label{automatic}
The Hanoi sequence ${\cal S}_{\infty}$ is pure morphic. It is the 
iterative fixed point of the $2$-morphism $\varphi$ on 
$\{a, b, c, \overline{a}, \overline{b}, \overline{c}\}^*$ defined by
$$
\varphi(a) := a \overline{c}, \ \ \varphi(b) := c \overline{b}, \ \ 
\varphi(c) := b \overline{a}, \ \ \varphi(\overline{a}) := a c, \ \ 
\varphi(\overline{b}) := c b, \ \ \varphi(\overline{c}) := b a. 
$$
In particular the sequence ${\cal S}_{\infty}$ is $2$-automatic.
\end{theorem}

\begin{remark}
Using the automaton-based
formulation of Theorem~\ref{automatic} above (see, e.g.,
\cite{AD}), it is possible to prove that the $j$th move in the algorithm 
for the optimal solution of the Tower of Hanoi
can be determined from the binary expansion of $j$, hence ``with 
eyes closed'' (i.e., without looking at the towers), and with bounded memory 
(the total needed memory is essentially remembering the morphism above, 
which does not depend on the number of disks).

It may also be worth noting that the Tower of Hanoi sequence is {\em squarefree};
it contains no block of moves $w$ immediately followed by another occurrence
of the same block \cite{AARS}. Also note that this is not the case for all
variations on this puzzle: for example the lazy Tower of Hanoi sequence
(see Remark~\ref{variations} and Theorem~\ref{allsap}) is not squarefree, since 
it begins with $a \ b \ a \ \overline{b} \ \overline{a} \ b \ a \ b \ a \cdots$.

\end{remark}

\section{Another ``mechanical'' way of generating the sequence of moves in 
${\cal S}_{\infty}$}

We begin with an informal definition (for more details the reader can 
look at \cite{AB} and the references therein).
{\em Toeplitz sequences} or {\em Toeplitz transforms} of sequences are obtained 
by starting from a periodic sequence on an alphabet 
$\{a_1, a_2, \ldots, a_r, \diamond\}$, where $\diamond$ is a marked 
symbol called ``hole''. Then all the holes in the sequence are replaced 
in order by the terms of a periodic sequence with holes on the same 
alphabet (possibly the same sequence). The process is iterated. If none
of the periodic sequences used in the construction begins with a
$\diamond$, the process converges and yields a {\em Toeplitz sequence}.

A classical example is the paperfolding sequence that results from
folding a strip of paper on itself infinitely many times, and from
looking at the pattern of
up and down creases after unfolding (see, e.g., \cite[p.~155]{AS}). 
This sequence can also be constructed by a Toeplitz transform as 
follows. Start with the $3$-letter alphabet $\{0, 1, \diamond\}$.
Take the sequence 
$(0 \ \diamond \ 1 \ \diamond)^{\infty} := 
0 \ \diamond \ 1 \diamond \ 0 \ \diamond \ 1 \ \diamond \ 0 \ \cdots$.
Replace the sub-sequence of diamonds by the sequence 
$(0 \ \diamond \ 1 \ \diamond)^{\infty}$ itself, and iterate the process.
This yields successively
$$
\begin{array}{llllllllllllllllllll}
&0 \ &\diamond \ &1 &\diamond \ &0 \ &\diamond \ &1 \ &\diamond \ &0 \ 
&\diamond \ &1 &\diamond \ &0 \ &\diamond \ &1 \ &\cdots
&=& &(0 \ \diamond \ 1 \ \diamond)^{\infty}
\\
&0 \ &0 \ &1 \ &\diamond \ &0 \ &1 \ &1 \ &\diamond \ &0 \
&0 \ &1 \ &\diamond \ &0 \ &1 \ &1 \ &\cdots
&=& &(0 \ 0 \ 1 \ \diamond \ 0 \ 1 \ 1 \ \diamond)^{\infty}
\\
&0 \ &0 \ &1 \ &0 \ &0 \ &1 \ &1 \ &\diamond \ &0 \
&0 \ &1 \ &1 \ &0 \ &1 \ &1 \ &\cdots
&=& &(0 \ 0 \ 1 \ 0 \ 0 \ 1 \ 1 \ \diamond \
0 \ 0 \ 1 \ 1 \ 0 \ 1 \ 1 \ \diamond)^{\infty}
\\
&\vdots
\\
\end{array}
$$
After having applied this process an infinite number of times, there
are no $\diamond$ left. The limit sequence is equal to the paperfolding
sequence
$$
0 \ 0 \ 1 \ 0 \ 0 \ 1 \ 1 \ 0 \ 0 \
0 \ 1 \ 1 \ 0 \ 1 \ 1 \ \cdots .
$$

The Hanoi sequence ${\cal S}_{\infty}$ can be constructed in a
similar way. Take the $7$-letter alphabet 
$\{a, b, c, \overline{a}, \overline{b}, \overline{c}, \diamond\}$.
Start with the sequence $(a \ \overline{c} \ b \ \diamond \ c \ \overline{b} 
\ a \ \diamond \ b \ \overline{a} \ c \ \diamond)^{\infty}$.
Replace the sequence of holes by the sequence itself. Then iterate the 
process. This gives sequences with ``fewer and fewer holes'' and ``more
and more coinciding with'' the Hanoi sequence ${\cal S}_{\infty}$, namely
$$
\begin{array}{cccccccccccccccccccccccccccccccccccccc}
&a &\overline{c} &b &\diamond &c &\overline{b} &a &\diamond 
&b &\overline{a} &c &\diamond &a &\overline{c} &b &\diamond 
&c &\overline{b} &a &\diamond &b &\overline{a} &c &\diamond 
&a &\overline{c} &b &\diamond &c &\overline{b} &a &\diamond 
&b &\overline{a} &c &\diamond &\cdots 
\\
&a &\overline{c} &b &a &c &\overline{b} &a &\overline{c}
&b &\overline{a} &c &b &a &\overline{c} &b &\diamond  
&c &\overline{b} &a &c &b &\overline{a} &c &\overline{b}
&a &\overline{c} &b &a &c &\overline{b} &a &\diamond  
&b &\overline{a} &c &b &\cdots
\\
&a &\overline{c} &b &a &c &\overline{b} &a &\overline{c}
&b &\overline{a} &c &b &a &\overline{c} &b &a
&c &\overline{b} &a &c &b &\overline{a} &c &\overline{b}
&a &\overline{c} &b &a &c &\overline{b} &a &\overline{c}
&b &\overline{a} &c &b &\cdots
\\
&\vdots 
\end{array}
$$
The following theorem was proved in \cite{AD}.

\begin{theorem}
The infinite Hanoi sequence ${\cal S}_{\infty}$ is equal to the 
Toeplitz transform obtained by starting from the sequence
$(a \ \overline{c} \ b \ \diamond \ c \ \overline{b} \ a \ \diamond \
b \ \overline{a} \ c \ \diamond)^{\infty}$, replacing the $\diamond$
by the elements of the sequence itself, then iterating the process
an infinite number of times.
\end{theorem}

\section{Classical sequences hidden behind the Hanoi sequence}

Several classical sequences are linked to the Hanoi sequence. 
We will describe some of them in this section.

\subsection{Period-doubling sequence}\label{per-dou}

A binary sequence ${\cal T}$ can be deduced from ${\cal S}_{\infty}$ by
replacing each of $a$, $b$, $c$ by $1$ and each of $\overline{a}$, 
$\overline{b}$, $\overline{c}$ by $0$ (i.e., the non-barred letters by 
$1$ and the barred letters by $0$), thus obtaining
$$
\begin{array}{lllllllllllllllllll}
{\cal S}_{\infty} &=& &a \ &\overline{c} \ &b \ &a \ &c \ &\overline{b} \ &a \
&\overline{c} \ &b \ &\overline{a} \ &c \ &b \ &a \ &\overline{c} \ &b \ &\cdots \\
{\cal T} &=& &1 \ &0 \ &1 \ &1 \ &1 \ &0 \ &1 \ &0 \ &1 \ &0 \ &1 \ &1 \ &1  \
&0 \ &1 \ &\cdots \\
\end{array}
$$
It is not difficult to prove \cite{AD} that ${\cal T}$ is the iterative
fixed point of the morphism $\omega$ on $\{0, 1\}^*$ defined by $\omega(1) := 10$,
$\omega(0) := 11$. This iterative fixed point is known as the {\em
period-doubling
sequence}. It was introduced in the study of iterations of unimodal continuous
functions in relation with Feigenbaum cascades
(see, e.g., \cite[pp.~176, 209]{AS}).

\subsection{Double-free subsets}\label{dou-fre}
    
Define a sequence of integers ${\cal U}$ by counting for each term of
${\cal S}_{\infty}$ the cumulative number of the
non-barred letters up to this term
$$
\begin{array}{lllllllllllllllllll}
{\cal S}_{\infty} &=& &a \ &\overline{c} \ &b \ &a \ &c \ &\overline{b} \ &a \
&\overline{c} \ &b \ &\overline{a} \ &c \ &b \ &a \ &\overline{c} \ &b \ &\cdots \\
{\cal U} &=& &1 \ &1 \ &2 \ &3 \ &4 \ &4 \ &5 \ &5 \ &6 \ &6 \ &7 \ &8 \ &9  \
&9 \ &10 \ &\cdots \\
\end{array}
$$
The sequence ${\cal U}$ is equal to the sequence of maximal sizes of a subset
$S$ of $\{1, 2, \ldots, n\}$ with the property that if $x$ is in $S$ then $2x$
is not.  (The sequences ${\cal T}$ in Subsection~\ref{per-dou} above and
${\cal U}$ are respectively called {\tt A035263} and {\tt A050292} in
Sloane's {\em Encyclopedia} \cite{Sloane},
where it is mentioned that {\tt A050292} is the summatory
function of {\tt A035263}).

\subsection{Prouhet-Thue-Morse sequence} \label{ptm}

Reducing the sequence ${\cal U}$ in Subsection~\ref{dou-fre} modulo $2$
(or, equivalently, taking the summatory function modulo $2$ of the sequence 
${\cal T}$ in Subsection~\ref{per-dou} above) yields a sequence ${\cal V}$
$$
\begin{array}{lllllllllllllllllll}
{\cal U} &=& &1 \ &1 \ &2 \ &3 \ &4 \ &4 \ &5 \ &5 \ &6 \ &6 \ &7 \ &8 \ &9  \
&9 \ &10 \ &\cdots \\
{\cal V} &=& &1 \ &1 \ &0 \ &1 \ &0 \ &0 \ &1 \ &1 \ &0 \ &0 \ &1 \ &0 \ &1  \
&1 \ &0 \ &\cdots \\
\end{array}
$$
which is the celebrated (Prouhet-)Thue-Morse sequence (see, e.g., \cite{ubi}),
deprived of its first $0$. Recall that the Prouhet-Thue-Morse sequence can
be defined as the unique iterative fixed point, beginning with $0$, of the 
morphism defined by $0 \to 01$, $1 \to 10$.

\subsection{Other classical sequences related to Hanoi}

Other classical sequences are related to the Tower of Hanoi or to 
its variations. 
We mention here the Sierpi\'nski gasket \cite{Hinz5},
the Pascal triangle \cite{Hinz2} (see also \cite{Hinz4} and 
the references therein), the Stern diatomic sequence \cite{Hinz4},
but also the Stirling numbers of the second kind \cite{KMP1}, 
the second order Eulerian numbers, the Lah numbers, and the Catalan 
numbers \cite{KMP2}.

\section{Variations on the Tower of Hanoi and morphisms}

As indicated in Remark~\ref{variations}, several variations on the 
Tower of Hanoi can be found in the literature. We will first indicate a 
generalization of Theorem~\ref{automatic}. Then we will give a new result
on the classical Tower of Hanoi and one of its avatars, and a new result 
on automatic sequences.

\subsection{Tower of Hanoi with restricted moves}

There are exactly five variations deduced ``up to isomorphism'' from the 
classical Hanoi puzzle by restricting the permitted moves, i.e., by 
forbidding some fixed subset of the set of moves 
$\{a, b, c, \overline{a}, \overline{b}, \overline{c}\}$; see \cite{Sapir}. 
The following result was proven in \cite{AllSap} (also see \cite{ABS, AD}
for the classical case and \cite{AD} for the cyclic case).

\begin{theorem}\label{allsap}
The five restricted Tower of Hanoi problems
give rise to infinite morphic sequences 
of moves, whose appropriate truncations describe the transfer of any given 
number of disks. Furthermore two of these infinite sequences are actually 
automatic sequences, namely the classical Hanoi sequence and the lazy Hanoi
sequence, which are, respectively, $2$-automatic and $3$-automatic. 
\end{theorem}

We give below the examples of the cyclic Tower of Hanoi
and the lazy Tower of Hanoi (defined above in Remark~\ref{variations}).

\begin{itemize}

\item 
There exists an infinite sequence over
the alphabet $\{a, b, c\}$ that is the common
limit of finite minimal sequences of moves for the {\em cyclic} Tower
of Hanoi that 
allow us to transfer $N$ disks from peg~$I$ to peg~$II$ or from peg~$I$ to peg~$III$. 
Furthermore this sequence is morphic: it is the image under the map ($1$-uniform morphism)
$F: \{f, g, h, u, v, w\} \to \{a, b, c\}$ where $F(f) = F(w) := a$,
$F(g) = F(u) := c$, $F(h) = F(v) := b$ of the iterative fixed point of the morphism 
$\psi$ defined on $\{f, g, h, u, v, w\}$ by
$$
\begin{array}{llll}
&\psi(f) := fvf, \ \ &\psi(g) := gwg, \ \ &\psi(h) := huh, \\
&\psi(u) := fg, \ \ &\psi(v) := gh, \ \ &\psi(w) := hf
\end{array}
$$

\item
the {\em lazy\,} Hanoi sequence is the iterative fixed point beginning with $a$ 
of the morphism $\lambda$ defined on $\{a, b, \overline{a}, \overline{b}\}^*$ by
$$
\lambda(a) := a \ b \ a, \ \ \lambda(\overline{a}) := a \ b \ \overline{a}, \ \
\lambda(b) := \overline{b} \ \overline{a} \ b, \ \
\lambda(\overline{b}) := \overline{b} \ \overline{a} \ \overline{b}.
$$
In particular it is $3$-automatic.
\end{itemize}

\subsection{New results}

\begin{definition}
We say a sequence is {\em non-uniformly pure morphic} if it is the iterative
fixed point of a non-uniform morphism. We say that a sequence is
{\em non-uniformly morphic} if it is the image (under a $1$-uniform morphism) 
of a non-uniformly pure morphic sequence. 
\end{definition}
For example, the sequence $abaababa \cdots$ generated by the morphism 
$a \rightarrow ab$, $b \rightarrow a$, is non-uniformly pure morphic. This sequence 
is known as the {\em (binary) Fibonacci sequence}, since it is also equal to the limit 
of the sequence of words $(u_n)_{n \geq 0}$ defined by $u_0 := a$, 
$u_1 := ab$, $u_{n+2} := u_{n+1} u_n$ for each $n \geq 0$. (Of course, as 
M. Mend\`es France pointed out, we certainly assume that the alphabet of the 
non-uniform morphism involved in the above definition is the same as the minimal 
alphabet of its fixed point. For example the morphism $0 \to 01$, $1 \to 10$, 
$2 \to 1101$, whose iterative fixed point is the Thue-Morse sequence, does not 
make that sequence non-uniformly morphic.)

Although most of the non-uniformly morphic sequences are not automatic (e.g.,
the binary Fibonacci sequence is not automatic), some sequences can be 
simultaneously automatic and non-uniformly morphic. An example is the 
sequence ${\cal Z}$ formed by the lengths of the strings of $1$'s between 
two consecutive zeros in the Thue-Morse sequence $0 {\cal V}$ (whose 
definition is recalled in Subsection~\ref{ptm}).
$$
\begin{array}{lll}
0 {\cal V} &=& 0 \ 1 \ 1 \ 0 \ 1 \ 0 \ 0 \ 1 \ 1 \ 0 \ 0 \ 1 \ 0 \ 
                1 \ 1 \ 0 \ \cdots \\
           &=& 0 \ (1 1) \ 0 \ (1) \ 0 \ ( \ ) \  0 \ (1 1) \ 0 \ 
               ( \ ) 0 \ (1) \ 0 \ (1 1) \ 0 \ \cdots \\ 
{\cal Z}   &=& 2 \ 1 \ 0 \ 2 \ 0 \ 1 \ 2 \ \cdots \\
\end{array}
$$
The sequence ${\cal Z}$ is both \cite{Berstel} the iterative fixed 
point of the morphism $2 \to 210, \ \ 1 \to 20, \ \ 0 \to 1$
and the image under the map $x \to x \bmod 3$ of the iterative fixed point
of the $2$-morphism $2 \to 21, \ \ 1 \to 02, \ \ 0 \to 04, \ \ 4 \to 20$.

We have just seen that the five variations of the Tower of Hanoi 
(with restricted moves) are morphic; two of them are actually 
automatic (namely the classical and the lazy Tower of Hanoi), and 
one is not (the cyclic Tower of Hanoi, see \cite{All}). It is asked 
in \cite{AllSap} whether it is true that the other two variations 
are not $k$-automatic for any $k \geq 2$.
Reversing that question in some sense,
we could instead ask whether the classical Hanoi sequence and the 
lazy Hanoi sequence (which are, respectively, $2$-automatic and 
$3$-automatic) are {\em also} non-uniformly morphic.
The following result seems to be new.

\begin{theorem}\label{newHanoi}
The classical Hanoi sequence is non-uniformly pure morphic:
it is the iterative fixed point of the (non-uniform) morphism $\xi$ defined on 
$\{a, b, c, \overline{a}, \overline{b}, \overline{c}\}^*$ by
$$
\begin{array}{llll}
&\xi(a) := a \overline{c} b, \ \ &\xi(b) := \overline{b}, \ \ 
&\xi(c) := \overline{a} c, \\
&\xi(\overline{a}) := a c b, \ \ &\xi(\overline{b}) := b, \ \
&\xi(\overline{c}) := a c. \\
\end{array}
$$
The lazy Hanoi sequence is non-uniformly pure morphic:
it is the iterative fixed point of the (non-uniform) morphism $\eta$ 
defined on $\{a, b, \overline{a}, \overline{b}\}^*$ by
$$
\begin{array}{lll}
&\eta(a) := a \ b \ a \ \overline{b}, \ \ &\eta(b) := \overline{a} \ b, \\
&\eta(\overline{a}) := a \ b \ \overline{a} \ \overline{b}, \ \
&\eta(\overline{b}) := \overline{a} \ \overline{b}. \ \
\end{array}
$$
\end{theorem}

\begin{proof}
We begin with the classical Hanoi sequence ${\cal S}_{\infty}$. An easy 
computation shows the following five equalities (where $\varphi$ is the 
morphism defined in Theorem~\ref{automatic}).
$$
\begin{array}{lll}
\xi(a \overline{c} b) &=& \varphi(a \overline{c} b), \\
\xi(a c \overline{b}) &=& \varphi(a c \overline{b}), \\
\xi(\overline{a} c b) &=& \varphi(\overline{a} c b), \\
\xi(a c b)            &=& \varphi(a c b), \\ 
\xi(\overline{a} c \overline{b}) &=& \varphi(\overline{a} c \overline{b}). \\
\end{array}
$$
Now, grouping the elements by triples, i.e., writing the sequence 
${\cal S}_{\infty}$ as
$$
{\cal S}_{\infty} = (a \ \overline{c} \ b) \ (a \ c \ \overline{b}) \ 
(a \ \overline{c} \ b ) \ (\overline{a} \ c \ b) \ (a \ \overline{c} \ b) \ \cdots,
$$
we know that only the five triples $a \overline{c} b$, $a c \overline{b}$,
$\overline{a} c b$, $a c b$, and $\overline{a} c \overline{b}$ occur 
(see \cite[Theorem~2]{Hinz3}, where this result is used to construct
a square-free sequence on a $5$-letter alphabet, starting with the classical
Hanoi sequence). Thus
$$
\begin{array}{lll}
\xi({\cal S}_{\infty}) 
&=&
\xi((a \ \overline{c} \ b) \ (a \ c \ \overline{b}) \ (a \ \overline{c} \ b ) \ 
(\overline{a} \ c \ b) \ (a \ \overline{c} \ b) \ \cdots) \\
&=&
\xi(a \ \overline{c} \ b) \ \xi(a \ c \ \overline{b}) \ \xi(a \ \overline{c} \ b ) \ 
\xi(\overline{a} \ c \ b) \ \xi(a \ \overline{c} \ b) \ \cdots \\
&=&
\varphi(a \ \overline{c} \ b) \ \varphi(a \ c \ \overline{b}) \ 
\varphi(a \ \overline{c} \ b ) \ \varphi(\overline{a} \ c \ b) \ 
\varphi(a \ \overline{c} \ b) \ \cdots \\
&=&
\varphi((a \ \overline{c} \ b) \ (a \ c \ \overline{b}) \ (a \ \overline{c} \ b ) \ 
(\overline{a} \ c \ b) \ (a \ \overline{c} \ b) \ \cdots) \\
&=&
\varphi({\cal S}_\infty) = {\cal S}_\infty. \\
\end{array}
$$

Now we look at the lazy Hanoi sequence. Using the morphism $\lambda$
defined in the second example following Theorem~\ref{allsap}, we note that
$$
\begin{array}{lllll}
\eta(a b a \overline{b}) &=& \lambda(a b a \overline{b}), \ \
\eta(a b \overline{a} \overline{b}) &=& \lambda(a b \overline{a} \overline{b}), \\ 
\eta(a \overline{b} \overline{a} b) &=& \lambda(a \overline{b} \overline{a} b), \ \
\eta(a \overline{b} \overline{a} \overline{b}) &=& 
\lambda(a \overline{b} \overline{a} \overline{b}), \\
\eta(\overline{a} b a b) &=& \lambda(\overline{a} b a b), \ \
\eta(\overline{a} \overline{b} a b) &=& \lambda(\overline{a} \overline{b} a b), \\ 
\eta(\overline{a} \overline{b} \overline{a} b) &=&
\lambda(\overline{a} \overline{b} \overline{a} b), \ \ 
\eta(\overline{a} \overline{b} \overline{a} \overline{b}) &=& 
\lambda(\overline{a} \overline{b} \overline{a} \overline{b}). \\
\end{array}
$$
Grouping as above the elements of the lazy Hanoi sequence by quadruples, 
we can write that sequence, say ${\cal H}_{\infty}$ as
$$
{\cal H}_{\infty} = a \ b \ a \ \overline{b} \ \overline{a} \ b \ a \ b \ a \
\overline{b} \ \overline{a} \ \overline{b} \ a \ b \ \cdots 
= (a \ b \ a \ \overline{b}) \ (\overline{a} \ b \ a \ b) \ 
(a \ \overline{b} \ \overline{a} \ \overline{b}) \ (a \ b \cdots 
$$
It is not hard to see that the $4$-letter blocks (also called {\em $4$-letter 
factors}) that can occur in this parenthesized version of ${\cal H}_{\infty}$ 
are among all $4$-letter blocks of ${\cal H}_{\infty}$ beginning with $a$ or 
$\overline{a}$. These blocks are necessarily subblocks of images by $\lambda$ 
of $2$-letter blocks of ${\cal H}_{\infty}$, i.e., of blocks 
$\{ab, a\overline{b}, \overline{a}b, \overline{a}\overline{b},
ba, b \overline{a}, \overline{b}a, \overline{b}\overline{a}\}$. 
Hence these $4$-letter blocks are among 
$a b a \overline{b}$, $a b \overline{a} \overline{b}$, 
$a \overline{b} \overline{a} b$, $a \overline{b} \overline{a} \overline{b}$,
$\overline{a} b a b$, $\overline{a} \overline{b} a b$,
$\overline{a} \overline{b} \overline{a} b$,
$\overline{a} \overline{b} \overline{a} \overline{b}$.
We then have
$$
\begin{array}{lll}
\eta({\cal H}_{\infty}) &=& \eta(a \ b \ a \ \overline{b} \ \overline{a} \ 
b \ a \ b \ a \ \overline{b} \ \overline{a} \ \overline{b} \ a \ b \ 
\overline{a} \ \overline{b} \ \cdots) \\
&=& \eta((a \ b \ a \ \overline{b}) \ (\overline{a} \ b \ a \ b) \
(a \ \overline{b} \ \overline{a} \ \overline{b}) \ 
(a \ b \ \overline{a} \ \overline{b}) \ \cdots ) \\
&=& \eta(a \ b \ a \ \overline{b}) \ \eta(\overline{a} \ b \ a \ b) \
\eta(a \ \overline{b} \ \overline{a} \ \overline{b}) \ 
\eta(a \ b \ \overline{a} \ \overline{b}) \ \cdots \\
&=& \lambda(a \ b \ a \ \overline{b}) \ \lambda(\overline{a} \ b \ a \ b) \
\lambda(a \ \overline{b} \ \overline{a} \ \overline{b}) \ 
\lambda(a \ b \ \overline{a} \ \overline{b}) \ \cdots  \\
&=& \lambda((a \ b \ a \ \overline{b}) \ (\overline{a} \ b \ a \ b) \
(a \ \overline{b} \ \overline{a} \ \overline{b}) \ 
(a \ b \ \overline{a} \ \overline{b}) \ \cdots ) \\
&=& \lambda({\cal H}_{\infty}) = {\cal H}_{\infty}. \mbox{ \qed}\\
\end{array}
$$
\end{proof}

\section{A little more on automatic sequences}

Automatic sequences, such as the classical and the lazy Hanoi sequences,
have numerous properties, in particular number-theoretical properties.
We refer, e.g., to \cite{AS}. We give here a characteristic property
of formal power series on a finite field in terms of automatic sequences
Before doing this let us consider again the period-doubling sequence 
${\cal T}$, i.e., the iterative fixed point of the morphism $\omega$ 
on $\{0, 1\}^*$ defined (see Section~\ref{per-dou}) 
by $\omega(1) := 10$, $\omega(0) := 11$. Let us 
identify $\{0, 1\}^*$ with the field of $2$ elements, ${\mathbb F}_2$.
The definitions of ${\cal T} = (t_n)_{n \geq 0}$ and of $\omega$ show
that, for every $n \geq 0$, $t_{2n} = 1$, $t_{2n+1} = 1 + t_n$.
Hence, denoting by $F$ the formal power series $\sum t_n X^n$ in
${\mathbb F}_2[[X]]$, we have (remember we are in characteristic $2$)
$$
\begin{array}{lll}
F &:=& \displaystyle \sum_{n \geq 0} t_n X^n = 
\sum_{n \geq 0} t_{2n} X^{2n} + \sum_{n \geq 0} t_{2n+1} X^{2n+1} = 
\sum_{n \geq 0} X^{2n} + \sum_{n \geq 0} (1+t_n) X^{2n+1} \\
&=& \displaystyle\frac{1+X}{1-X^2} + X \sum_{n \geq 0} t_n X^{2n} 
= \displaystyle\frac{1}{1-X} + 
X \left(\sum_{n \geq 0} t_n X^n\right)^2 = \frac{1}{1-X} + X F^2.
\end{array}
$$
Remembering once more that we are in characteristic two, this can be 
written
$$
X(1+X)F^2 + (1+X)F + 1 = 0
$$
which means that $F$ is algebraic (at most quadratic -- actually exactly 
quadratic since the sequence ${\cal T}$ is not ultimately periodic) on 
the field of rational functions ${\mathbb F}_2(X)$. This result is a
particular case of a more general result that we give hereafter (see 
\cite{Christol, CKMFR}).

\begin{theorem}
Let ${\mathbb F}_q$ be the finite field of cardinality $q$. Let
$(a_n)_{n \geq 0}$ be a sequence on ${\mathbb F}_q$. Then, the
formal power series $\sum a_n X^n$ is algebraic over the field
of rational functions ${\mathbb F}_q(X)$ if and only if the
sequence $(a_n)_{n \geq 0}$ is $q$-automatic.
\end{theorem}

\section{Conclusion}

The Tower of Hanoi and its variations have many mathematical properties.
They also 
are used in cognitive psychology. It is interesting to note that psychologists,
as well as mathematicians, are looking at the shortest path to reconstruct the 
tower on another peg. But, is really the shortest path the most interesting?  If 
the answer is yes in terms of strategy for a puzzle, it is not clear that the answer
is also yes for detecting all kinds of skills for a human being. One of the speakers
at the Symposium 
\selectlanguage{french}
``La \og Tour d'Hano\"{\i} \fg, un casse-t\^ete 
math\'ematique d'\'Edouard Lucas (1842-1891)'' 
\selectlanguage{english}
said that finding 
the shortest path might not be the most interesting question. It reminded the
first author (JPA) of a discussion he once 
had with the French composer M. Fr\'emiot. After JPA showed him the Tower of 
Hanoi and an algorithmic solution, Fr\'emiot composed a {\em Messe pour orgue 
\`a l'usage des paroisses} \cite{Fremiot}. Interestingly
enough, what he emphasized is the rule ``no disk on a smaller one''.
His ``Messe'' used only that rule, without any attempt to reconstruct the tower
with the smallest number of moves. In contrast to algorithms (robots?) that
(try to) optimize a quantitative criterion, or to mathematicians
who (try to) prove the optimality of a solution or study the set of all
solutions, is it not the case that composers, and more generally, artists are 
interested in qualitatively (rather than quantitatively) exceptional elements 
of a given set, in ``jewels'' rather than in ``generic'' elements, in 
non-necessarily rational choices rather than in exhaustive studies or rigorous
proofs...? Ren\'e Char wrote ``Le po\`ete doit laisser des traces de son passage
non des preuves. Seules les traces font r\^ever''.

\end{document}